\documentclass[12pt,reqno]{amsart}

\usepackage{amssymb}

\newcommand{\seq}{\subseteq}
\newcommand{\stm}{\setminus}

\newcommand{\alp}{\alpha}
\newcommand{\del}{\delta}
\newcommand{\eps}{\varepsilon}
\newcommand{\lam}{\lambda}
\renewcommand{\phi}{\varphi}
\newcommand{\sig}{\sigma}
\newcommand{\zet}{\zeta}

\newcommand{\C}{{\mathbb C}}
\newcommand{\F}{{\mathbb F}}
\newcommand{\K}{{\mathbb K}}
\newcommand{\Q}{{\mathbb Q}}
\newcommand{\R}{{\mathbb R}}
\newcommand{\Z}{{\mathbb Z}}

\newcommand{\cO}{{\mathcal O}}
\newcommand{\cR}{{\mathcal R}}
\newcommand{\cN}{{\mathcal N}}

\newcommand{\fI}{{\mathfrak I}}

\newcommand{\fq}{{\mathfrak q}}
\newcommand{\fp}{{\mathfrak p}}
\newcommand{\fP}{{\mathfrak P}}

\renewcommand{\>}{\rangle}
\renewcommand{\(}{\left(}
\renewcommand{\)}{\right)}

\DeclareMathOperator{\tr}{tr}
\DeclareMathOperator{\ord}{ord}
\DeclareMathOperator{\lcm}{lcm}
\DeclareMathOperator{\Gal}{Gal}

\renewcommand{\=}{\stackrel{!}{=}}

\newcommand{\Krt}{{\K^{\raisebox{1.5pt}{$\scriptscriptstyle\surd$}}}}

\newcommand{\refl}[1]{\ref{l:#1}}
\newcommand{\refm}[1]{\ref{m:#1}}
\newcommand{\reft}[1]{\ref{t:#1}}
\newcommand{\refc}[1]{\ref{c:#1}}
\newcommand{\refp}[1]{\ref{p:#1}}
\newcommand{\refs}[1]{\ref{s:#1}}
\newcommand{\refb}[1]{\cite{b:#1}}
\newcommand{\refe}[1]{\eqref{e:#1}}

\newtheorem{claim}{Claim}
\newtheorem{lemma}{Lemma}
\newtheorem{theorem}{Theorem}
\newtheorem{corollary}{Corollary}
\newtheorem{proposition}{Proposition}

\title{Quadratic residues and difference sets}

\author{Vsevolod F. Lev}
\address{Department of Mathematics, The university of Haifa at Oranim,
  Tivon 36006, Israel}
\email{math@haifa.ac.il}

\author{Jack Sonn}
\address{Department of Mathematics,
  Technion --– Israel Institute of Technology, Haifa 32000, Israel}
\email{sonn@math.technion.ac.il}

\subjclass[2010]%
  {Primary: 11B13;   
   Secondary: 11A15, 
              11B34, 
              11P70, 
              11T21, 
              05B10} 

\begin{document}
\baselineskip = 16pt

\begin{abstract}
It has been conjectured by S\'ark\"ozy that with finitely many
exceptions, the set of quadratic residues modulo a prime $p$ cannot be
represented as a sumset $\{a+b\colon a\in A, b\in B\}$ with non-singleton
sets $A,B\seq\F_p$. The case $A=B$ of this conjecture has been recently
established by Shkredov. The analogous problem for differences remains
open: is it true that for all sufficiently large primes $p$, the set of
quadratic residues modulo $p$ is not of the form
 $\{a'-a''\colon a',a''\in A,\,a'\ne a''\}$ with $A\seq\F_p$?

We attack here a presumably more tractable variant of this problem, which is
to show that there is no $A\seq\F_p$ such that every quadratic residue has a
\emph{unique} representation as $a'-a''$ with $a',a''\in A$, and no
non-residue is represented in this form. We give a number of necessary
conditions for the existence of such $A$, involving for the most part the
behavior of primes dividing $p-1$. These conditions enable us to rule out all
primes $p$ in the range $13<p<10^{18}$ (the primes $p=5$ and $p=13$ being
conjecturally the only exceptions).

\end{abstract}

\maketitle

\section{Background and Motivation}\label{s:history}

S\'ark\"ozy \refb{sa} conjectured that the set $\cR_p$ of all quadratic
residues modulo a prime $p$ is not representable as a sumset
 $\{a+b\colon a\in A,\,b\in B\}$, whenever $A,B\seq\F_p$ satisfy
$\min\{|A|,|B|\}>1$. Shkredov \refb{sh} has recently established the
particular case $B=A$ of this conjecture, showing that
 $\{a'+a''\colon a',a''\in A\}\ne\cR_p$, except if $p=3$ and $A=\{2\}$. He
has also proved that $\cR_p$ cannot be represented as a \emph{restricted
sumset}: $\{a'+a''\colon a',a''\in A,\,a'\ne a''\}\ne\cR_p$ for $A\seq\F_p$,
with several exceptions for $p\le 13$.

The argument of \refb{sh} does not seem to extend to handle differences
(instead of sums) and to show that
\begin{equation}\label{e:difnotres}
  \{a'-a''\colon a',a''\in A,\,a'\ne a''\}\ne\cR_p,\quad A\seq\F_p.
\end{equation}
We notice that for equality to hold in \refe{difnotres}, one needs to have
$2\binom{|A|}2\ge |\cR_p|$, which readily yields
\begin{equation}\label{e:Aislarge}
  |A| > \sqrt{p/2}.
\end{equation}
At the same time, there is a famous, long-standing conjecture saying that for
every $\eps>0$, if $A\seq\F_p$ has the property that $a'-a''\in\cR_p$ for all
$a',a''\in A$ with $a'\ne a''$, then
\begin{equation}\label{e:Aissmall}
  |A| < p^\eps
\end{equation}
provided that $p$ is sufficiently large. (We refer the reader to \refb{sh}
for several more related conjectures and discussion.) Combining
\refe{Aislarge} and \refe{Aissmall}, one immediately derives that
\refe{difnotres} is true for all but finitely many primes $p$.

Unfortunately, the conjecture just mentioned is presently out of reach, and
neither could we prove \refe{difnotres}. As a step in this direction, we
investigate the following, presumably easier, problem:

\smallskip
\begin{quotation}
\emph{Does there exist a subset $A\seq\F_p$ such that the differences
$a'-a''$ with $a',a''\in A$, $a'\ne a''$, list all quadratic residues modulo
$p$, and every quadratic residue is listed exactly once?}
\end{quotation}

Even this question does not eventually receive a complete answer. However, we
were able to establish a number of necessary conditions, and use them to show
that in the range $13<p<10^{18}$, there are no ``exceptional primes''. This
makes it extremely plausible to conjecture that no such primes exist at all,
with just two exceptions $p=5$ and $p=13$ addressed below.

\section{Summary of Results}\label{s:summary}

In this section we introduce basic notation and present our results. Most of
the proofs are postponed to subsequent sections; see the ``proof locator'' at
the very end of the section.

Recall, that for a prime $p$ we denote by $\F_p$ the finite field of order
$p$, and by $\cR_p$ the set of all quadratic residues modulo $p$. We also
denote by $\cN_p$ the set of all quadratic non-residues modulo $p$, to have
the decomposition $\F_p=\cR_p\cup\cN_p\cup\{0\}$.

For subsets $A$ and $S$ of an additively written abelian group, the notation
$A-A\=S$ will indicate that every element of $S$ has a unique representation
as a difference of two elements of $A$ and, moreover, every such non-zero
difference belongs to $S$. (In our context, the underlying group is always
the additive group of the field $\F_p$, and $S$ is one of the sets $\cR_p$
and $\cN_p$.) Our goal is thus to show that, with few exceptions,
\begin{equation}\label{e:neverholds}
  A-A \= \cR_p
\end{equation}
does not hold.

One immediate observation is that for \refe{neverholds} to hold, letting
$n:=|A|$, one needs to have $n(n-1)=\frac{p-1}2$; that is, $p=2n(n-1)+1$. As
a result, $p\equiv 1\pmod 4$ --- a conclusion which also follows by observing
that the set of all differences $a'-a''$ is symmetric, whence $\cR_p$ must be
symmetric too.

Experimenting with small values of $p$, one finds two remarkable
counterexamples to \refe{neverholds}: namely, the sets $A_5:=\{2,3\}\seq\F_5$
and $A_{13}:=\{2,5,6\}\seq\F_{13}$. Clearly, all affinely equivalent sets of
the form $\{\mu a+c\colon a\in A_p\}$, where $\mu\in\cR_p$ and $c\in\F_p$ are
fixed parameters (and $p\in\{5,13\}$) work too, and it is not difficult to
see that no other sets $A$ satisfying \refe{neverholds} exist for $p\le 13$;
indeed, we believe that there are no more such sets at all.

What makes the two sets $A_5$ and $A_{13}$ special? An interesting feature
they have in common is that both of them are cosets of a subgroup of the
multiplicative group of the corresponding field; indeed, $A_{13}$ is a coset
of the subgroup $\{1,3,9\}<\F_{13}^\times$, while $A_5$ is a coset of the
subgroup $\{1,4\}<\F_5^\times$. In addition, $A_5$ is affinely equivalent to
the set $\{0,1\}$, which is a union of $0$ and a subgroup of $\F_5^\times$.
Our first two theorems show that constructions of this sort do not work for
$p>13$.

\begin{theorem}\label{t:NoGP1}
For a prime $p>13$, there is no coset $A=gH$, with $H<\F_p^\times$ and
$g\in\F_p^\times$, such that $A-A\=\cR_p$.
\end{theorem}

\begin{theorem}\label{t:NoGP2}
For a prime $p>5$, there is no coset $gH$, with $H<\F_p^\times$ and
$g\in\F_p^\times$, such that, letting $A:=gH\cup\{0\}$, we have $A-A\=\cR_p$.
\end{theorem}

For integer $\mu$ and a subset $A$ of an additively written abelian group, by
$\mu A$ we denote the dilate of $A$ by the factor of $\mu$:
  $$ \mu A:=\{\mu a\colon a\in A\}. $$
Extending slightly one of the central notions of the theory of difference
sets, we say that $\mu$ is a \emph{multiplier} of $A$ if $\mu$ is co-prime
with the exponent of the group, say $e$, and there exists a group element $g$
such that $\mu A=A+g$. Clearly, in this case every integer from the residue
class of $\mu$ modulo $e$ is also a multiplier of $A$. This shows that the
multipliers of a given set $A$ can be considered as elements of the group of
units $(\Z/e\Z)^\times$, and it is immediately seen that they actually form a
subgroup; we denote this subgroup by $M_A$, and call it the \emph{multiplier
subgroup} of $A$.

It is readily seen that all translates of a subset $A$ of an abelian group
have the same multiplier subgroup. If, furthermore, $|A|$ is co-prime with
the exponent $e$ of the group, then there is a translate of $A$ whose
elements add up to $0$. Denoting this translate by $A_0$ and observing that
 $\mu A_0=A_0+g$ implies $g=0$ (as follows by comparing the sums of
elements of each side), we conclude that if $\gcd(|A|,e)=1$, then $A$
has a translate which is fixed by every multiplier $\mu\in M_A$.

Here we are interested in the situation where the underlying group has prime
order. In this case, every subset $A$ has a translate fixed by its multiplier
subgroup $M_A$. This translate is then a union of several cosets of $M_A$
and, possibly, the zero element of the group. Consequently, using
multipliers, Theorems~\reft{NoGP1} and \reft{NoGP2} can be restated as
follows: if $p>13$ and $A\seq\F_p$ satisfies $A-A\=\cR_p$, then choosing
$g\in\F_p$ so that the elements of the translate $A-g$ add up to $0$, the set
$(A-g)\stm\{0\}$ is a union of at least two cosets of $M_A$.

Our next result shows, albeit in a rather indirect way, that ``normally'', a
set $A\seq\F_p$ satisfying $A-A\=\cR_p$ must have a large multiplier
subgroup.

For a prime $p\equiv 1\pmod 4$, let $G_p$ denote the greatest common divisor
of the orders modulo $p$ of all primes dividing $\frac{p-1}4$:
  $$ \textstyle
      G_p := \gcd \left\{ \ord_p(q)\colon q\mid\frac{p-1}4,
                                         \ q\ \text{is prime} \right\}. $$
\begin{theorem}\label{t:gcdtest}
If $p$ is a prime and $A\seq\F_p$ satisfies $A-A\=\cR_p$, then the multiplier
subgroup $M_A$ lies above the order-$G_p$ subgroup of $\F_p^\times$;
equivalently, $|M_A|$ is divisible by $G_p$.
\end{theorem}

The quantity $G_p$ is difficult to study analytically, but one can expect
that it is usually quite large: for, if $r^v\mid p-1$ with $r$ prime and
$v>0$ integer, then in order for $r^v$ not to divide $G_p$, there must be a
prime $q\mid\frac{p-1}4$ which is a degree-$r$ residue modulo $p$, the
``probability'' of which for every specific $q$ is $1/r$. Computations show
that, for instance, among all primes $p\le 10^{12}$ of the form
$p=2n(n-1)+1$, there are less than 1.4\% those satisfying $G_p<\sqrt p$.

Recalling that $A-A\=\cR_p$ implies $p=2n(n-1)+1$ with $n=|A|$, from
Theorem~\reft{gcdtest} and in view of Theorems~\reft{NoGP1} and~\reft{NoGP2}
we get
\begin{corollary}\label{c:gcdtest}
Suppose that $p$ is a prime. If there exists a subset $A\seq\F_p$ with
$A-A\=\cR_p$ then, writing $p=2n(n-1)+1$, either $G_p$ is a proper divisor of
$n$, or $G_p$ is a proper divisor of $n-1$.
\end{corollary}
To give an impression of how strong Corollary~\refc{gcdtest} is, we remark
that it sieves out over 99.7\% of all primes $p=2n(n-1)+1$ with $p<10^{12}$.

For integer $k\ge 1$, let $\Phi_k$ denote the $k$\,th cyclotomic polynomial.
Yet another useful consequence of Theorem~\reft{gcdtest} is
\begin{corollary}\label{c:cyclotomic}
Let $p$ be a prime, and suppose that there exists a subset $A\seq\F_p$ with
$A-A\=\cR_p$. If an element $z\in\F_p$ and an integer $k\ge 2$ satisfy
$\ord_p(z)\nmid k$ and $\ord_p(z)\mid G_p$, then $\Phi_k(z)\in\cR_p$.
\end{corollary}

The practical implication of Corollary~\refc{cyclotomic} is that if we can
find a residue $z\in\F_p$ of degree $\frac{p-1}{G_p}$ and an integer $k\ge 2$
such that $z^k\ne1$ and $\Phi_k(z)\in\cN_p$, then there is no set $A\seq\F_p$
with $A-A\=\cR_p$.

To prove Corollary~\refc{cyclotomic}, denote by $H$ the order-$G_p$ subgroup
of $\F_p^\times$, and consider the differences $h'-h''$ with $h',h''\in H$,
$h'\ne h''$. By Theorem~\reft{gcdtest}, either all these differences are
quadratic residues, or they all are quadratic non-residues. If
 $\ord_p(z)\mid G_p$ and $\ord_p(z)\nmid k$, then both $z$ and $z^k$ are
non-unit elements of $H$, and consequently either both $z-1$ and $z^k-1$ are
quadratic residues, or they both are quadratic non-residues. In either case,
  $$ \prod_{\substack{d\mid k \\ d>1}} \Phi_d(z)
                                          = \frac{z^k-1}{z-1} \in \cR_p, $$
and the claim follows by induction on $k$.

It is somewhat surprising that if a set $A\seq\F_p$ with $A-A\=\cR_p$ exists,
then all orders $\ord_p(q)$ appearing in the definition of the quantity $G_p$
are odd.
\begin{theorem}\label{t:paritytest}
Let $p$ be a prime. If there exists a subset $A\seq\F_p$ satisfying
$A-A\=\cR_p$, then for every prime $q\mid\frac{p-1}4$, the order $\ord_p(q)$
is odd.
\end{theorem}

\begin{corollary}\label{c:paritytest}
Let $p$ be a prime. If there exists a subset $A\seq\F_p$ satisfying
$A-A\=\cR_p$, then writing $p=2n(n-1)+1$ we have $n\equiv 2\pmod 4$ or
$n\equiv 3\pmod 4$; hence, $p\equiv 5\pmod 8$.
\end{corollary}

To derive Corollary~\refc{paritytest} from Theorem~\reft{paritytest}, observe
that if we had $n\equiv 0\pmod 4$ or $n\equiv 1\pmod 4$, then $\frac{p-1}4$
were even and, consequently, $\frac{p-1}4$ and $p-1$ would have same prime
divisors. As a result, all prime divisors of $p-1$ would be of odd order
modulo $p$, which is impossible as $p-1$ itself has even order.

Using a biquadratic reciprocity law due to Lemmermeyer \refb{le}, from
Theorem~\reft{paritytest} we will derive
\begin{theorem}\label{t:divtest}
Let $p$ be a prime. If there exists a subset $A\seq\F_p$ satisfying
$A-A\=\cR_p$ then, writing $p=2n(n-1)+1$, neither $n$ nor $n-1$ have prime
divisors congruent to $7$ modulo $8$. Moreover, of the two numbers $n$ and
$n-1$, the odd one has no prime divisors congruent to $5$ modulo $8$, and the
even one has no prime divisors congruent to $3$ modulo $8$.
\end{theorem}

Computations show that there are very few primes passing both the test of
Corollary~\refc{gcdtest} and that of Theorem~\reft{divtest}. In the range
$13<p<10^{18}$, there are only five such primes, corresponding to the values
of $n$ listed in the following table:

\bigskip
\renewcommand{\arraystretch}{1.15}
\begin{center}
\begin{tabular}{ r | c | c | l }
          $n$ & \, $\del$\, & $(n-\del)/G_p$ & $n-1,\ n$ \\ \hline
           51 & 1 & 2 & $2\cdot 5^2,\ 3\cdot 17$ \\
          650 & 0 & 2 & $11\cdot 59,\ 2\cdot 5^2\cdot 13$ \\
        32283 & 1 & 2 & $2\cdot 16141,\ 3^2\cdot 17\cdot 211$ \\
     57303490 & 1 & 3 & $3\cdot 1579\cdot 12097,\ 2\cdot5\cdot 5730349$ \\
    377687811 & 0 & 3 & $2\cdot 5\cdot 17\cdot 113\cdot 19661,\ 3\cdot 1787\cdot 70451$
\end{tabular}

\bigskip\small{
{\sc Fig.~1}. The second column gives the  value of $\del\in\{0,1\}$ such
that $G_p\mid n-\del$, \\ \qquad\quad the last column contains the prime
decompositions of $n-1$ and $n$.}
\end{center}

Every individual value of $n$ in the table is easy to rule out using
Corollary~\refc{cyclotomic}. For instance, the first exceptional value $n=51$
corresponds to the prime $p=5101$; since $(5101-1)/G_{5101}=204$, applying
Corollary~\refc{cyclotomic} with $k=2$ we conclude that if $A\seq\F_{5101}$
satisfying $A-A\=\cR_{5101}$ existed, then every degree-$204$ residue
$z\in\F_p$ with $z^2\ne 1$ would satisfy $z+1\in\cR_{5101}$; this conclusion,
however, is violated for $z=2^{204}$.

The remaining four exceptional cases can be dealt with in an analogous
way; say, one can take $z=2^{(p-1)/G_p}$ for $n=650$ and $n=377687811$,
and $z=3^{(p-1)/G_p}$ for $n=32283$ and $n=57303490$ (with $k=2$ in each
case). We thus conclude that there are no primes $13<p<10^{18}$ for which
$A\seq\F_p$ with $A-A\=\cR_p$ exists.

Theorem~\reft{paritytest} will be derived as a straightforward corollary of
the Semi-primitivity Theorem from the theory of difference sets. Recall, that
for positive integer $v,k$, and $\lam$, a \emph{$(v,k,\lam)$-difference set}
is a $k$-element subset of a $v$-element group such that (assuming additive
notation) every non-zero group element has exactly $\lam$ representations as
a difference of two elements of the set. The following somewhat unexpected
claim shows how difference sets come into the play, and allows us to apply
the well-established machinery of difference sets in our problem.
\begin{claim}\label{m:diff}
Suppose that $p$ is a prime and $A\seq\F_p$ satisfies $A-A\=\cR_p$. Write
$n:=|A|$ and fix arbitrarily a quadratic non-residue $\nu\in\cN_p$. Then the
$n^2$ sums $a'+\nu a''$ with $a',a''\in A$ are pairwise distinct, and the set
$D$ of all these sums is a $(p,n^2,n(n+1)/2)$-difference set in $\F_p$.
\end{claim}

We remark that the Multiplier Conjecture \cite[Conjecture~6.7]{b:la} along
with Claim~\refm{diff} lead to a conclusion much stronger than
Corollary~\refc{gcdtest}: namely, if there is a subset $A\seq\F_p$ with
$A-A\=\cR_p$, then, writing $p=2n(n-1)+1$, the \emph{least common multiple}
$\lcm\left\{\ord_p(q)\colon q\mid\frac{p-1}4\right\}$ is a divisor of either
$n$ or $n-1$.

On a historical note, it was Broughton \refb{b} who first used biquadratic
reciprocity to study $(2n(n-1)+1,n^2,n(n+1)/2)$-difference sets.

Our last result is a lemma which is used in the proof of
Theorems~\reft{NoGP1} and~\reft{NoGP2}, and which we believe is also
of independent interest.
\begin{lemma}\label{l:MAodd}
If $p>5$ is a prime and $A\seq\F_p$ satisfies $A-A\=\cR_p$, then $|M_A|$ is
odd; that is, $-1\notin M_A$.
\end{lemma}

The rest of the paper is devoted to the proofs of the above-discussed
results. We prove Lemma~\refl{MAodd} in the next section, and
Theorems~\reft{NoGP1} and \reft{NoGP2} in Section~\refs{NoGP}. In
Section~\refs{diff} we prove Claim~\refm{diff}, present the Semi-primitivity
Theorem, and derive Theorem~\reft{paritytest}. In Section~\refs{divtest} we
state Lemmermeyer's biquadratic reciprocity law and prove
Theorem~\reft{divtest}. Theorem~\reft{gcdtest} is proved in
Section~\refs{gcdtest}; the proof uses some basic algebraic number theory.
Finally, in the Appendix we give an equivalent restatement of the problem
studied in this paper in terms of algebraic number theory.

\section{$|M_A|$ is odd: the proof of Lemma~\refl{MAodd}}\label{s:MAodd}

Suppose that $p$ is a prime and $A\seq\F_p$ satisfies $A-A\=\cR_p$; we want
to show that the multiplier subgroup $M_A<\F_p^\times$ has odd order.

For a subset $S\seq\F_p$ and integer $j\ge 0$, let
  $$ \sig_j(S)=\sum_{s\in S} s^j, $$
subject to the agreement that if $0\in S$ and $j=0$, then the corresponding
summand is equal to $1$ (so that $\sig_0(S)=|S|$). For every $1\le k<(p-1)/2$
we have
  $$ \sum_{a',a''\in A} (a'-a'')^k = \sum_{x\in\cR_p} x^k = 0; $$
expanding the binomial and changing the order of summation, we get
\begin{equation}\label{e:tmp26}
  \sum_{j=0}^k (-1)^j \binom kj \sig_j(A) \sig_{k-j}(A) = 0.
\end{equation}
Write $m:=|M_A|$. Having $A$ suitably translated, we can assume that
$A\stm\{0\}$ is a union of cosets of $M_A$, and let then $C$ be the set of
arbitrarily chosen representatives of these cosets. We distinguish two cases.

Suppose first that $0\notin A$. In this case $\sig_j(A)=\sig_j(C)\sig_j(M_A)$
and
\begin{equation*}
   \sig_j(M_A) = \begin{cases}
                m\ &\text{if}\ m\mid j, \\
                0\ &\text{otherwise},
              \end{cases}
\end{equation*}
whence \refe{tmp26} is non-trivial only if $m\mid k$, and in this case (with
a minor change of notation) it can be re-written as
\begin{equation}\label{e:tmp27}
   \sum_{j=0}^k (-1)^{jm} \binom{km}{jm} \sig_{jm}(C) \sig_{(k-j)m}(C) = 0,
                                          \quad 0<k<\frac{p-1}{2m}.
\end{equation}
Taking $k=1$ gives $(1+(-1)^m)\sig_0(C)\sig_m(C)=0$, and if $m$ were even
(contrary to the assertion of the lemma) then, in view of $\sig_0(C)=|C|\ne
0$, we would have $\sig_m(C)=0$. Furthermore, we could then re-write
\refe{tmp27} as
  $$ 2|C|\sig_{km}(C)
         = -\sum_{j=1}^{k-1} \binom{km}{jm} \sig_{jm}(C) \sig_{(k-j)m}(C), $$
and substituting subsequently $k=2,3,\ldots$ we conclude that
$\sig_{km}(C)=0$ whenever $0<k<(p-1)/(2m)$. Equivalently, the $|C|$ elements
$c^m\ (c\in C)$ have the property that the sum of their $k$th powers vanish
for all $0<k<(p-1)/(2m)$; hence for all $0<k\le|C|$ in view of
  $$ |C| = \frac{|A|}{|M_A|} = \frac nm < \frac{n(n-1)}m = \frac{p-1}{2m}. $$
(we use here our standard notation: $n=|A|$ and $p=2n(n-1)+1$. Notice that
this estimate uses the assumption $p>5$.) As a result, all these elements,
and therefore also all elements of $C$, are equal to $0$, a contradiction
establishing the assertion in the case $0\notin A$.

Turning to the situation where $0\in A$, we write $A_0:=A\stm\{0\}$ and
notice that in this case $\sig_0(A)=|A|=m|C|+1$ and
$\sig_j(A)=\sig_j(A_0)=\sig_j(C)\sig_j(M_A)$ for every $j>0$; as a result,
  $$ \sig_j(A) = \begin{cases}
                m|C|+1\ &\text{if}\ j=0, \\
                m\sig_j(C)\   &\text{if $m\mid j$ and $j>0$}, \\
                0\         &\text{if $m\nmid j$}.
              \end{cases} $$
Hence, assuming that $m$ is even, from \refe{tmp26} we get
  $$ 2(m|C|+1)\cdot m\sig_{km}(C)
       = -m^2 \sum_{j=1}^{k-1} \binom{km}{jm}\sig_{jm}(C)\sig_{(k-j)m}(C),
                                 \quad 0<k<\frac{p-1}{2m}. $$
Now taking $k=1$ yields $\sig_m(C)=0$, and then subsequently $\sig_{km}(C)=0$
for each $0<k<(p-1)/(2m)$, leading to a contradiction exactly as above.

This completes the proof of Lemma~\refl{MAodd}.

\section{Proofs of Theorems~\reft{NoGP1} and~\reft{NoGP2}:
  One Coset is not Enough}\label{s:NoGP}

For a prime $p$, let $\chi_p$ denote the quadratic character modulo $p$
extended onto the whole field $\F_p$ by $\chi_p(0)=0$. We need the following
well-known identity 
(which is equivalent, for instance, to \cite[Chapter~5, Exercise~8]{b:ir}):
\begin{equation}\label{e:sum-chi}
  \sum_{x\in\F_p} \chi_p((x+a)(x+b)) =
       \begin{cases}
         p-1 &\text{if}\ a=b, \\
         -1  &\text{if}\ a\ne b,
       \end{cases} \qquad a,b\in\F_p.
\end{equation}
Recall, that we are interested in the situation where $p\equiv 1\pmod 4$, in
which case $\chi_p(-1)=1$; equivalently, $\chi_p(-x)=\chi_p(x)$ for all
$x\in\F_p$.

\begin{proof}[Proof of Theorem~\reft{NoGP1}]
Clearly, it suffices to show that for $p>13$ prime and $H<\F_p^\times$, one
cannot have $H-H\=\cR_p$ or $H-H\=\cN_p$. For a contradiction, suppose that
one of these relations holds true. Write $n:=|H|$, so that $p=2n(n-1)+1$.
From Lemma~\refl{MAodd} (as applied to a suitable coset of $H$ in the case
$H-H\=\cN_p$), we know that $n$ is odd, implying $-1\notin H$; hence, $H$ is
disjoint with $-H:=\{-h\colon h\in H\}$.

For any $h_1,h_2\in H$ with $h_1\ne h_2$, either both $h_1^2-h_2^2$ and
$h_1-h_2$ are quadratic residues, or they both are quadratic non-residues. In
either case, their quotient $h_1+h_2$ is a quadratic residue; that is,
\begin{equation}\label{e:2HinQ}
  \chi_p(h_1+h_2)=1,\quad h_1,h_2\in H,\ h_1\ne h_2.
\end{equation}
We distinguish two cases, according to whether $H-H\=\cR_p$ or $H-H\=\cN_p$.

Suppose first that $H-H\=\cR_p$, and let in this case
  $$ \sig(x) := \sum_{h\in H} \big( \chi_p(x+h) + \chi_p(x-h) \big),
                                                          \quad x\in\F_p. $$
In view of \refe{2HinQ} and our present assumption $H-H\=\cR_p$, for each
$x\in H$ we have
  $$ \sig(x) \ge (n-2) + (n-1) = 2n-3. $$
Along with $\sig(-x)=\sig(x)$ (following from $p\equiv 1\pmod 4$ and
$\chi_p(-1)=1$ resulting from it), this yields
\begin{equation}\label{e:overB1}
  \sum_{x\in H\cup(-H)} \sig^2(x) \ge 2n(2n-3)^2.
\end{equation}
On the other hand, the sum extended over \emph{all} $x\in\F_p$ can be
computed explicitly:
\begin{align}\label{e:complsum1}
  \sum_{x\in\F_p} \sig^2(x)
     &= \sum_{x\in\F_p} \sum_{h_1,h_2\in H}
              \big( \chi_p(x+h_1)+\chi_p(x-h_1) \big)
                             \big( \chi_p(x+h_2)+\chi_p(x-h_2) \big) \notag \\
     &= \sum_{h_1,h_2\in H} \sum_{x\in\F_p}
               \big( \chi_p((x+h_1)(x+h_2)) + \chi_p((x-h_1)(x-h_2))  \notag \\
     &\hspace{1.4in}
             + \chi_p((x+h_1)(x-h_2)) + \chi_p((x-h_1)(x+h_2)) \big) \notag \\
     &= 2pn-4n^2 \notag \\
     &= 2n(2n^2-4n+1),
\end{align}
as it follows from \refe{sum-chi} and since $h_1\ne-h_2$ whenever
 $h_1,h_2\in H$ in view of $-1\notin H$. Comparing \refe{overB1} and
\refe{complsum1} we conclude that $2n(2n-3)^2\le2n(2n^2-4n+1)$, which
simplifies to $(n-2)^2\le 0$ and thus yields $n=2$, contrary to the
assumption $p>13$.

Addressing now the case where $H-H\=\cN_p$, we re-define the sum $\sig(x)$
letting this time
  $$ \sig(x) := \sum_{h\in H} \big(\chi_p(x+h)-\chi_p(x-h)\big),
                                                           \quad x\in\F_p. $$
In view of \refe{2HinQ} and the assumption $H-H\=\cN_p$, we have again
  $$ \sig(x) \ge (n-2)+(n-1)=2n-3,\quad x\in H. $$
Since $\sig(-x)=-\sig(x)$, we derive that
  $$ \sum_{x\in H\cup(-H)}\sig^2(x) \ge 2n(2n-3)^2. $$
On the other hand, a computation similar to \refe{complsum1} gives
  $$ \sum_{x\in\F_p} \sig^2(x) = 2pn = 2n(2n^2-2n+1). $$
As a result, $2n(2n-3)^2 \le 2n(2n^2-2n+1)$, leading to $n\le 4$. To complete
the proof we notice that $n\le 3$ correspond to $p\le 13$, while $n=4$ yields
$p=25$, which is composite.
\end{proof}

\begin{proof}[Proof of Theorem~\reft{NoGP2}]
The proof is a variation of that of Theorem~\reft{NoGP1}.

Aiming at a contradiction, suppose that $p>5$ is prime, $H<\F_p^\times$,
$g\in\F_p^\times$, and $A:=gH\cup\{0\}$ satisfies $A-A\=\cR_p$. Since $g$ is
representable as a difference of two elements of $A$, we have $g\in\cR_p$,
and dilating $A$ by the factor $g^{-1}$ we can assume that, indeed, $g=1$;
that is, $A=H\cup\{0\}$.

Write $n:=|A|$, so that $p=2n(n-1)+1$ and $|H|=n-1$. From Lemma~\refl{MAodd},
we know that $|H|$ is odd, whence $-1\notin H$ and therefore $H$ is disjoint
with $-H$.

For any $h\in H$ and $a_1,a_2\in A$ with $a_1\ne a_2$, both $a_1h-a_2h$ and
$a_1-a_2$ are quadratic residues, and so must be their quotient $h$; thus,
\begin{equation}\label{e:HinQ2}
  \chi_p(h)=1,\quad h\in H.
\end{equation}
Similarly,
\begin{equation}\label{e:2HinQ2}
  \chi_p(h_1+h_2)=1,\quad h_1,h_2\in H,\ h_1\ne h_2
\end{equation}
in view of $h_1+h_2=(h_1^2-h_2^2)/(h_1-h_2)$.

Let
  $$ \sig(x) := \sum_{a\in A} \big(\chi_p(x+a)+\chi_p(x-a)\big),
                                                      \quad x\in\F_p. $$
From \refe{HinQ2} and \refe{2HinQ2}, and since $A-A\=\cR_p$, we have
  $$ \sig(x) \ge (n-2)+(n-1)=2n-3,\quad x\in H $$
and
  $$ \sig(0) = 2(n-1). $$
Observing that $\sig(-x)=\sig(x)$ we derive that
  $$ \sum_{x\in H\cup(-H)\cup\{0\}}\sig^2(x)
                      \ge 2(n-1)(2n-3)^2+4(n-1)^2 = 2(n-1)(4n^2-10n+7). $$
On the other hand, a computation similar to \refe{complsum1} gives
  $$ \sum_{x\in\F_p} \sig^2(x) = 2(n+1)p-4n^2 = 2(n-1)(2n^2-1). $$
As a result, $4n^2-10n+7\le 2n^2-1$, implying $n\le 4$. The assumption $p>5$
now gives $n=3$; consequently, $p=13$ and $|H|=2$, whence $H=\{1,-1\}$.
However, the set $A=\{0,1,-1\}\seq\F_{13}$ does not have the property
$A-A\=\cR_{13}$.
\end{proof}

\section{Proofs of Claim~\refm{diff} and Theorem~\reft{paritytest}}%
  \label{s:diff}

\begin{proof}[Proof of Claim~\refm{diff}]
To see that the sums $a'+\nu a''$ are pairwise distinct, we notice that
$a_1'+\nu a_1''=a_2'+\nu a_2''$ with $(a_1',a_1'')\ne(a_2',a_2'')$ would
result in $\nu=(a_1'-a_2')/(a_2''-a_1'')$, while for
$a_1',a_1'',a_2',a_2''\in A$, both the numerator and the denominator are
quadratic residues in view of $A-A\=\cR_p$.

It remains to show that every non-zero element of $\F_p$ has exactly
$n(n+1)/2$ representations as a difference of two elements of the set
$D:=\{a'+\nu a''\colon a',a''\in A\}$.

Let $\zet$ be a fixed primitive root of unity of degree $p$, and denote by
$\K$ the $p$th cyclotomic field; that is, $\zet\ne\zet^p=1$ and
$\K=\Q[\zet]$. Write $\alp:=\sum_{a\in A}\zet^a$, so that $A-A\=\cR_p$ yields
\begin{equation}\label{e:abs^2}
  |\alp|^2=n+\rho,
\end{equation}
where
\begin{equation}\label{e:rho}
  \rho := \sum_{x\in\cR_p} \zet^x = \frac{\sqrt p-1}2
\end{equation}
is a quadratic Gaussian period (see, for instance, \cite[Chapter~3,]{b:d}).

Set $\del:=\sum_{d\in D}\zet^d$; thus,
\begin{equation}\label{e:del}
  \del = \sum_{a'\in A} \zet^{a'} \cdot \sum_{a''\in A} \zet^{\nu a''}
                                                       = \alp\phi(\alp),
\end{equation}
with $\phi\in\Gal(\K/\Q)$ defined by $\phi(\zet)=\zet^\nu$. Let
$\tau\in\Gal(\K/\Q)$ denote the complex conjugation automorphism. Since
$\Gal(\K/\Q)$ is abelian (\cite[Chapter~13, \S2,~Corollary~2]{b:ir} or
\cite[Page~18, Corollary~2]{b:m}), we have
\begin{equation}\label{e:alpsquared}
  \phi(|\alp|^2) = \phi(\alp\tau(\alp))
                           = \phi(\alp)\tau(\phi(\alp)) = |\phi(\alp)|^2.
\end{equation}
From \refe{abs^2}--\refe{alpsquared} and
  $$ \phi(\rho) = \sum_{x\in\cR_p} \zet^{\nu x}
          = \sum_{x\in\cN_p} \zet^x = -1 - \sum_{x\in\cR_p} \zet^x
                                                             = -1 - \rho, $$
we obtain
\begin{multline*}
  |\del|^2=|\alp|^2|\phi(\alp)|^2=|\alp|^2\phi(|\alp|^2) \\
                       = (n+\rho)(n-1-\rho)=\frac{n(n-1)}2
                        = |D| + \frac{n(n+1)}2 \sum_{x\in\F_p^\times} \zet^x.
\end{multline*}
Comparing this equality with
  $$ |\del|^2 = |D| + \sum_{x\in\F_p^\times} r(x)\zet^x, $$
where $r(x)$ is the number of representations of $x$ as a difference of two
elements of $D$, we conclude that $r(x)=n(n+1)/2$ for every
$x\in\F_p^\times$.
\end{proof}

We remark that the second assertion of Claim~\refm{diff} can also be proved
using the group ring approach. Namely, identifying subsets
$A,D,\cR_p,\cN_p,\F_p^\times\seq\F_p$ with the corresponding elements of the
group ring $\Z\F_p$, we have
  $$ D=AA^{(\nu)},\ AA^{(-1)}=n+\cR_p,\ \cR_p^{(\nu)}=\cN_p,
                             \ \text{and}\ \cR_p\cN_p=\frac{n(n-1)}2\,\F_p, $$
the last equality reflecting the well-known fact that for $p\equiv 1\pmod 4$,
every element of $\F_p^\times$ has exactly $\frac{p-1}4$ representations as a
sum of quadratic residue and a quadratic non-residue. Hence, we have the
chain of group ring equalities
\begin{multline*}
  DD^{(-1)} = AA^{(\nu)}A^{(-1)}A^{(-\nu)}
       = (n+\cR_p)(n+\cR_p)^{(\nu)} \\
          = (n+\cR_p)(n+\cN_p)
              = n^2 + n\F_p^\times + \frac{n(n-1)}2\,\F_p^\times
                 = n^2 + \frac{n(n+1)}2\,\F_p^\times,
\end{multline*}
proving the assertion.

We now state the part of the Semi-primitivity Theorem that is relevant for
our purposes. For co-prime integer $q,e\ge 1$, by $\<q\>_e$ we denote the
subgroup of $(\Z/e\Z)^\times$, multiplicatively generated by $q$.
\begin{theorem}[\protect{\cite[Theorem~4.5]{b:la}}]\label{t:st}
Suppose that $G$ is a finite abelian group of exponent $e$. If $G$ possesses
a $(v,k,\lam)$-difference set, then for any prime $q$ with $q\mid k-\lam$ and
$q\nmid e$, we have $-1\notin\<q\>_e$.
\end{theorem}

To deduce Theorem~\reft{paritytest} from Theorem~\reft{st}, we apply the
latter to the set $D$ of Claim~\refm{diff}. Since
  $$ n^2 - \frac{n(n+1)}2 = \frac{n(n-1)}2 = \frac{p-1}4, $$
we conclude that if $q\mid\frac{p-1}4$ is prime, then $\<q\>_p$ is an
odd-order subgroup of $\F_p^\times$; that is, $\ord_p(q)$ is odd. This proves
Theorem~\reft{paritytest}.

\section{Bi-quadratic reciprocity and the Proof of Theorem~\reft{divtest}}%
  \label{s:divtest}

The proof of Theorem~\reft{divtest} relies on Lemmermeyer's biquadratic
reciprocity law. To state it, we recall that the \emph{rational biquadratic
residue symbol} is defined for prime $p\equiv 1\pmod 4$ and quadratic residue
$b\in\cR_p$ by
  $$ \(\frac bp\)_4 = \begin{cases}
        1 &\ \text{if $b$ is a biquadratic residue modulo $p$}, \\
       -1 &\ \text{if $b$ is not a biquadratic residue modulo $p$}.
                      \end{cases} $$
Notice, that $\(b/p\)_4\equiv b^{\frac{p-1}4}\pmod p$ implies
multiplicativity of the rational biquadratic residue symbol.

For consistency, in this section we use the Legendre symbol $\(\cdot/p\)$ for
the quadratic character modulo $p$ (which was denoted $\chi_p(\cdot)$ in
Section~\refs{NoGP}, mostly for typographical reasons).

\begin{theorem}[\protect{\cite[Proposition~5.5]{b:le}}]\label{t:lemmer}
Suppose that $p\equiv 1\pmod 4$ is prime, and write $p=u^2+v^2$ with $u$ odd
and $v$ even. Suppose also that $q>2$ is a prime with $\(p/q\)=1$, and let
$c$ be an integer such that $c^2\equiv p\pmod q$. Finally, let
$q^\ast:=(-1)^{(q-1)/2}q$, so that $\(q^\ast/p\)=1$ by multiplicativity of
the Legendre symbol and the quadratic reciprocity law. Then
  $$ \( \frac{q^\ast}p\)_4
       = \begin{cases}
           \(\frac{c(v+c)}q\) &\ \text{if $q\nmid v+c$}, \\
           \(\frac 2q\) &\ \text{if $q\mid v+c$.}
         \end{cases} $$
\end{theorem}

We remark that, strictly speaking, the case where $q\mid v+c$ is not
addressed in \refb{le}, but it is easy to deduce from the case where $q\nmid
v+c$. For, if $q\mid v+c$, then $q\nmid v-c$ in view of $q\nmid c$, and
applying then the original Lemmermeyer's theorem with $c$ replaced by $-c$,
we get
  $$ \(\frac{q^*}p\)_4 = \(\frac{-c(v-c)}q\)
                                   = \(\frac{-c(-2c)}q\) = \(\frac 2q\). $$

\begin{proof}[Proof of Theorem~\reft{divtest}]
Suppose that $p$ is a prime and $A\seq\F_p$ satisfies $A-A\=\cR_p$; thus,
$p=2n(n-1)+1$ where $n:=|A|$. From Corollary~\refc{paritytest}, we have
$p\equiv 5\pmod 8$, whence
\begin{equation}\label{e:-1Gauss}
  \(\frac{-1}p\)_4 = (-1)^{\frac{p-1}4}=-1.
\end{equation}

Let $u$ and $v$ denote the odd and the even of the two numbers $n-1$ and $n$,
respectively; notice that this is consistent with the notation of
Theorem~\reft{lemmer} as $p=(n-1)^2+n^2=u^2+v^2$. Since $p\equiv 5\pmod 8$, a
prime $q$ divides $\frac{p-1}4=\frac12uv$ if and only if it is odd and
divides either $u$, or $v$. In this case $p\equiv 1\pmod q$, and we apply
Theorem~\reft{lemmer} with $c=1$ to obtain
\begin{equation}\label{e:qstar1}
  \(\frac{q^*}p\)_4
       = \begin{cases}
           \(\frac{v+1}q\) &\ \text{if $q\nmid v+1$}, \\
           \(\frac 2q\) &\ \text{if $q\mid v+1$,}
         \end{cases}
\end{equation}
where $q^\ast:=(-1)^{(q-1)/2}q$. On the other hand, Theorem~\reft{paritytest}
shows that $q$ is a biquadratic residue modulo $p$, and therefore using
\refe{-1Gauss} we get
\begin{equation}\label{e:qstar2}
  \left(\frac{q^*}{p}\right)_4
        = \left(\frac{(-1)^{(q-1)/2}}{p}\right)_4\left(\frac{q}{p}\right)_4
            = \left(\frac{-1}{q}\right)\left(\frac{q}{p}\right)_4
                                                        = \(\frac{-1}q\).
\end{equation}
From\refe{qstar1} and \refe{qstar2},
\begin{equation}\label{e:dag1}
  \left(\frac{v+1}{q}\right)=\left(\frac{-1}{q}\right)
                                                \quad \text{if}\ q\nmid v+1,
\end{equation}
and
\begin{equation}\label{e:dag2}
  \left(\frac{2}{q}\right)=\left(\frac{-1}{q}\right)
                                                 \quad \text{if}\ q\mid v+1.
\end{equation}
If $q\mid v$, then the former of these equalities immediately gives
$q\in\{1,5\}\pmod 8$. If $q\mid u$, we distinguish two further sub-cases:
$q\mid v+1$ and $q\nmid v+1$. If $q\mid v+1$, then \refe{dag2} gives
$q\in\{1,3\}\pmod 8$. If $q\nmid v+1$, then $u\in\{v-1,v+1\}$ along with our
present assumption $q\mid u$ show that $u=v-1$; thus, $q\mid v-1$, and
\refe{dag1} leads to the same conclusion $q\in\{1,3\}\pmod 8$ as above.

We have shown that for a prime $q>2$, if $q$ divides the even of the two
numbers $n-1$ and $n$, then $q\equiv 1\pmod 8$ or $q\equiv 5\pmod 8$, and
if $q$ divides the odd of these two numbers, then $q\equiv 1\pmod 8$ or
$q\equiv 3\pmod 8$. Thus is equivalent to the assertion of
Theorem~\reft{divtest}.
\end{proof}

\section{Proof of Theorem~\reft{gcdtest}: $M_A$ lies above the order-$G_p$
  subgroup of $\F_p^\times$}\label{s:gcdtest}

In this section and the Appendix we use several basic algebraic number theory
facts, such as for instance:
\begin{itemize}
\item[i)] the Galois group of the $m$th cyclotomic field is isomorphic to
    the group of units $(\Z/m\Z)^\times$; hence, it is abelian;
\item[ii)] if $p$ and $q$ are distinct odd primes, then, letting
    $f:=\ord_p(q)$, the principal ideal $(q)$ in the $p$th cyclotomic
    field splits into a product of $(p-1)/f$ pairwise distinct prime
    ideals, all of which are fixed by the order-$f$ subgroup of the
    corresponding Galois group;
\item[iii)] Kronecker's theorem: an algebraic integer all of whose
    algebraic conjugates lie on the unit circle is a root of unity;
    consequently, any cyclotomic integer of modulus $1$ is a root of
    unity;
\item[iv)] if $m$ is odd, then the only roots of unity of the $m$th
    cyclotomic field are the roots of degree $2m$.
\end{itemize}
The proofs can be found in any standard algebraic number theory textbook, as
\refb{ir} or \refb{m}.

\begin{proof}[Proof of Theorem~\reft{gcdtest}]
Suppose that $p$ is a prime and $A\seq\F_p$ satisfies $A-A\=\cR_p$. Write
$n:=|A|$, so that $p=2n(n-1)+1$. Let $\zet$ be a primitive root of unity of
degree $p$, and denote by $\K$ the $p$\,th cyclotomic field (thus,
$\K=\Q[\zet]$), and by $\cO$ the ring of integers of $\K$. As in the proof of
Claim~\refm{diff}, write $\alp:=\sum_{a\in A}\zet^a$, so that $\alp\in\cO$
and
\begin{equation}\label{e:abs^2-2}
  |\alp|^2=n+\rho
\end{equation}
with
\begin{equation}\label{e:rho-2}
  \rho := \sum_{x\in\cR_p} \zet^x = \frac{\sqrt p-1}2.
\end{equation}

It is well known that every rational prime $q\ne p$ splits in $\cO$ into a
product of $(p-1)/\ord_p(q)$ pairwise distinct  prime ideals, all of which
are fixed by the subgroup of $\Gal(\K/\Q)$ of order $\ord_p(q)$. The
intersection of these subgroups over all primes $q\mid \frac{p-1}4$ is the
subgroup $H\le\Gal(\K/\Q)$ of order $|H|=G_p$, and since, by \refe{abs^2-2},
$\alp$ is a divisor of $n+\rho$, which in turn is a divisor of
$\frac{p-1}4=(n+\rho)(n-1-\rho)$, we conclude that the ideal generated by
$\alp$ is fixed by $H$. Hence, for every automorphism $\phi\in H$ there
exists a unit $u\in\cO$ (depending on $\phi$) such that
\begin{equation}\label{e:ast}
  \phi(\alp)=u\alp.
\end{equation}

Since $p$ is a quadratic residue modulo every odd prime $q$ dividing $p-1$,
by quadratic reciprocity, $q$ is a quadratic residue modulo $p$; that is,
$q^{\frac{p-1}2}\equiv 1\pmod p$. This shows that $\ord_p(q)$ is a divisor of
$(p-1)/2$. As a result, $G_p$ divides $(p-1)/2$; that is, $H$ is contained in
the subgroup of order $(p-1)/2$, which is easily seen to have $\Q[\sqrt p]$
as its fixed field. Therefore, re-using equality \refe{alpsquared} from the
proof of Claim~\refm{diff} and in view of \refe{abs^2-2}, for every
automorphism $\phi\in H$ we have
  $$ |\phi(\alp)|^2 = \phi(|\alp|^2) = n+\rho = |\alp|^2. $$
Comparing this with \refe{ast}, we conclude that $|u|=1$. From the fact that
$\Gal(\K/\Q)$ is abelian it follows then that all algebraic conjugates of $u$
have modulus $1$, and by Kronecker's theorem $u$ is a root of unity; thus,
either $u=\zet^v$, or $u=-\zet^v$ with some $v\in\F_p$ depending on $\phi$.
The latter option is ruled out by considering traces from $\K$ to $\Q$: we
have $\tr(\phi(\alp))=\tr(\alp)$ and $\tr(-\zet^v\alp)\equiv-\tr(\alp)\pmod
p$, while $\tr(\alp)\equiv -n\not\equiv 0\pmod p$. Therefore,
\begin{equation}\label{e:ast-refined}
  \phi(\alp)=\zet^v\alp;\quad \phi\in H,\ v=v(\phi)\in\F_p.
\end{equation}

Recalling the definition of $\alp$ and identifying $\Gal(\K/\Q)$ with
$\F_p^\times$, we can interpret \refe{ast-refined} as saying that for every
$\phi\in H<\F_p^\times$, there exists $v=v(\phi)\in\F_p$ such that the dilate
$\phi A=\{\phi a\colon a\in A\}$ satisfies $\phi A=A+v$; that is, $\phi$ is a
multiplier of $A$.
\end{proof}

\appendix
\section*{Appendix: An algebraic number theory restatement}

We aim here to pursue a little further the algebraic approach that was
employed in the proofs of Claim~\refm{diff} and Theorem~\reft{gcdtest}, in
the hope that it can ultimately give more insights into the problem. We keep
using the notation introduced in these proofs: namely, given a prime $p$, we
denote by $\zet$ a fixed primitive root of unity of degree $p$, by $\K$ the
$p$\,th cyclotomic field, by $\cO$ the ring of integers of $\K$, and we let
$\rho:=(\sqrt p-1)/2$. By $\tr$ we denote the trace function from $\K$ to
$\Q$. Our goal is to prove the two following results.

\begin{proposition}\label{p:restatement}
Let $p$ be a prime number. For a subset $A\seq\F_p$ with $A-A\=\cR_p$ to
exist, it is necessary and sufficient that $p=2n(n-1)+1$ with an integer $n$,
and that there is an algebraic integer $\alp\in\cO$ such that
$|\alp|^2=n+\rho$ and $\tr(\alp\zet^{-k})\in\{-n,p-n\}$ for every integer
$k$.
\end{proposition}

\begin{proposition}\label{p:hasse}
Let $p$ be a prime of the form $p=2n(n-1)+1$ with $n$ an integer. For an
algebraic integer $\alp\in\cO$ with $|\alp|^2=n+\rho$ to exist, it is
necessary and sufficient that for every prime $q$ dividing $p-1$ to an odd
power, the order $\ord_p(q)$ is odd.
\end{proposition}

To prove Proposition~\refp{restatement}, we need
\begin{lemma}\label{l:indcond}
Let $p$ be a prime and $n\in[1,p-1]$ an integer. In order for $\alpha\in\cO$
to satisfy $\tr(\alp\zet^{-k})\in\{-n,p-n\}$ for every integer $k$, it is
necessary and sufficient that $\alp=\sum_{a\in A}\zet^a$, where $A$ is an
$n$-element subset of $\F_p$.
\end{lemma}

\begin{proof}
It is readily seen that the condition is sufficient: if
 $\alp=\sum_{a\in A}\zet^a$ with $A\seq\F_p$ and $|A|=n$, then
  $$ \tr(\alp\zet^{-k}) =
       \begin{cases}
         -n\ &\text{if}\  k\notin A, \\
         p-n\ &\text{if}\ k\in A.
       \end{cases} $$
To prove necessity, write $\alp=\sum_{x\in\F_p} a_x\zet^x$ with integer
coefficients $a_x$. For every $k\in\Z$ we have then
  $$ \tr(\alp\zet^{-k}) = pa_k-\sum_{x\in\F_p}a_x $$
(where $k$ in the right-hand side is identified with its canonical image in
$\F_p$), and the assumption $\tr(\alp\zet^{-k})\in\{-n,p-n\}$ implies that
the coefficients $a_x$ attain at most two distinct integer values. Since
adding simultaneously the same integer to all $a_x$ does not affect the value
of the sum $\sum_{x\in\F_p}a_x\zet^x$, we can assume without loss of
generality that actually at most one value assumed by $a_x$ is distinct from
$0$; hence, writing $A:=\{x\in\F_p\colon a_x\ne 0\}$, there is an integer $c$
such that
\begin{equation}\label{e:ast26}
  \alp = c\sum_{a\in A}\zet^a.
\end{equation}
In fact, the subset $A\seq\F_p$ is proper and non-empty and $c\ne 0$, as
otherwise we would have $\alp=0$ which is inconsistent with
$\tr(\alp\zet^{-k})\in\{-n,p-n\}$. Consequently, \refe{ast26} implies that
$\tr(\alp\zet^{-k})$ assumes exactly two distinct values, both divisible by
$c$. Observing, on the other hand, that $\gcd(-n,p-n)=\gcd(n,p)=1$, we
conclude that $c\in\{-1,1\}$. Replacing now $A$ with its complement in
$\F_p$, if necessary, we can assume that, indeed, $c=1$ holds. Thus,
$\alp=\sum_{a\in A}\zet^a$, and it remains to notice that this yields
$\tr(\alp\zet^{-k})\in\{-|A|,p-|A|\}$, whence $|A|=n$.
\end{proof}

\begin{proof}[Proof of Proposition~\refp{restatement}]
We know from Lemma~\refl{indcond} (see also the proofs of Claim~\refm{diff}
and Theorem~\reft{gcdtest}) that if $A-A\=\cR_p$ for a subset $A\seq\F_p$
then, writing $n:=|A|$ and $\alp:=\sum_{a\in A}\zet^a$, we have
$p=2n(n-1)+1$, $|\alp|^2=n+\rho$, and $\tr(\alp\zet^{-k})\in\{-n,p-n\}$ for
every integer $k$.

Conversely, suppose that $p=2n(n-1)+1$ and that for some $\alp\in\cO$ we have
$|\alp|^2=n+\rho$ and $\tr(\alp\zet^{-k})\in\{-n,p-n\}$ for every integer
$k$. By Lemma~\refl{indcond}, there is an $n$-element subset $A\seq\F_p$ such
that $\alp=\sum_{a\in A}\zet^a$. Hence,
  $$ \sum_{x\in\cR_p} \zet^x = \rho = |\alp|^2 - n
             = \sum_{\substack{a',a''\in A\\a'\ne a''}} \zet^{a'-a''}, $$
implying $A-A\=\cR_p$.
\end{proof}

\begin{proof}[Proof of Proposition~\refp{hasse}]
Consider a prime divisor $q$ of $p-1$ and denote by $v$ the power to which
$q$ divides $(p-1)/4$; thus, $v$ is either equal, or smaller by $2$ than the
power to which $q$ divides $p-1$. Since $p\equiv 1\pmod q$ and, consequently,
$p$ is a square mod $q$, if $q$ is odd, then it splits into two ideal primes
in $\Q(\sqrt p)$. This conclusion stays true also if $q=2$ and $v>0$: for, in
this case $p\equiv 1\pmod 8$ (see, for instance,
 \cite[Propositions~13.1.3 and~13.1.4]{b:ir} or
\cite[Chapter~3, Theorem~25]{b:m}). Now the decomposition
  $$ \frac{p-1}4=(n+\rho)(n-1-\rho) $$
and the fact that $n+\rho$ and $n-1-\rho$ are co-prime elements of
 $\Q(\sqrt p)$ show that the $v\,$th power of one of the two ideal primes
into which $q$ splits divides $n+\rho$, while the $v$\,th power of another
one divides $n-1-\rho$. Denote by $\fq$ the prime whose $v$\,th power divides
$n+\rho$; we thus have $(n+\rho)=\fq^v\fI$, where $\fI<\cO$ is an ideal
co-prime with $q$.

Write $f:=\ord_p(q)$, so that $q$ splits into $(p-1)/f$ pairwise distinct
ideal primes in $\cO$ and, accordingly, $\fq$ splits into $k:=(p-1)/(2f)$
pairwise distinct ideal primes: $\fq=\fq_1\ldots\fq_k$, where each $\fq_i$ is
stable under the subgroup $H<\Gal(\Q/\K)$ of order $f$. Assuming
$|\alp|^2=n+\rho$ and observing that $|\alp|^2=\alp\tau(\alp)$, where $\tau$
is the complex conjugation automorphism of $\K$, we thus have
\begin{equation}\label{e:qI}
  (\alp)\tau((\alp)) = \fq_1^v\ldots\fq_k^v\, \fI.
\end{equation}

Suppose now that $f$ is even, so that $\tau\in H$ and, consequently,
$\tau(\fq_i)=\fq_i$ for each $i\in[1,k]$. Comparing this with \refe{qI} we
conclude that the factor $\fq_i^v$ in its right-hand side must split evenly
between the two factors $(\alp)$ and $\tau((\alp))$; therefore, $v$ must be
even. This proves necessity.

To prove sufficiency we invoke the Hasse norm theorem
\cite[Theorem~V.4.5]{b:j} which says that if $K$ is a cyclic extension of a
number field $L$, then an element of $L$ is the norm (from $K$ to $L$) of an
element of $K$ if and only if it is a norm locally everywhere. The reader
will see that, in fact, the theorem also gives necessity; however, we prefer
to keep the simple ``elementary'' argument presented above.

Specified to our situation, Hasse's theorem gives the following. Let $\K^+$
be the real subfield of $\K$. For a prime ideal $\fp\subset\K^+$, denote by
$\K_\fp^+$ the completion of $\K^+$ at $\fp$, and by $\K_\fp$ the
corresponding completion of $\K$; thus, $\K_\fp=\K\K_\fp^+$. Then, according
to the Hasse theorem, $n+\rho$ is a norm from $\K$ to $\K^+$ if and only if
it is a norm from $\K_\fp$ to $\K_\fp^+$ for every prime $\fp$ of $\K^+$,
including the infinite primes.

Accordingly, let $\fp\subset\K^+$ be a prime. We first show that $n+\rho$ is
always a norm from $\K_\fp$ to $\K_\fp^+$ whenever $\fp\nmid\frac{p-1}{4}$.
For notational convenience, we write below $\Krt:=\Q(\sqrt{p})$.

If $\fp$ is an infinite prime, then it is a real prime and $\K_\fp^+$ is the
field $\R$ of real numbers, as $\K^+$ is totally real. Furthermore, every
real square, hence every positive real number, and in particular $n+\rho$, is
a norm from the quadratic extension $\K_\fp=\C$.

If $\fp$ is a finite prime dividing $p$, then it is unique with this
property, and $p$ is totally and tamely ramified in $\K$. Thus the extension
$\K_\fp/\K_\fp^+$ is a tamely ramified quadratic extension. Since $n+\rho$ is
not divisible by $\fp$, it is a unit in $\K_\fp^+$, so by
 \cite[Chapter~V, \S3, Proposition~5]{b:se} is a norm from  $\K_\fp$ if and
only if it is a square modulo $\fp$. As the residue field of $\K_\fp$ modulo
$\fp$ is $\F_p$, this is equivalent to $n+\rho$ being a square modulo the
uniformizer $\sqrt p$ of $\Krt\Q_p$, where $\Q_p$ is the field of $p$-adic
rationals, i.e.\ the completion of $\Q$ at $p$. Now
 $n+\rho\equiv n-\frac12\pmod{\sqrt p}$, with the congruence in (a
localization of) the ring of integers of $\Krt$. At the same time,
$p=2n(n-1)+1$ implies $n-\frac12\equiv n^2\pmod p$. It follows that
 $n-\frac12\equiv n^2\pmod{\sqrt p}$, hence $n+\rho\equiv n^2\pmod{\sqrt p}$,
and so $n+\rho\equiv n^2\pmod\fp$.

Finally, if $\fp$ is a finite prime not dividing $p$ (and also not dividing
$\frac{p-1}4$), then the extension $\K_\fp/\K_\fp^+$ is unramified, in which
case every unit of $\K_\fp^+$ is a norm from $\K_\fp$
 \cite[Chapter~V, \S2, Corollary to Proposition~3]{b:se}. But $n+\rho$ is a
unit of $\K_\fp^+$, as follows from the observation that
$N_{\Krt/\Q}(n+\rho)=\frac{p-1}4$ is not divisible by $\fp$.

We have thus shown that $n+\rho$ is always a norm from $\K_\fp$ to $\K_\fp^+$
whenever $\fp\nmid\frac{p-1}4$, and it remains to determine when $n+\rho$ is
a norm for the primes $\fp\mid\frac{p-1}4$. Fix such a prime $\fp\seq\K^+$,
and let $\fq$ be the prime in $\Krt$ lying below $\fp$, and $q$ be the
rational prime lying below $\fp$ and $\fq$. Also, let $\fq'$ be the conjugate
of $\fq$ over $\Q$; since $q$ splits into two primes in $\Krt$ (see the very
beginning of the proof for the explanation), we have the prime factorization
$q\cO_\Krt=\fq\fq'$.

Let $v_\fp,v_\fq,v_{\fq'}$, and $v_q$ be the valuations on $\K^+,\Krt,\Krt$,
and $\Q$, corresponding to $\fp,\fq,\fq'$, and $q$, respectively. Since $q$
is unramified in $\K$ (the only ramified prime in $\K$ is $p$), we may assume
that all these valuations are normalized; that is, their value groups are
$\Z$.

Trivially, $n+\rho$ is a norm from $\K_\fp$ to $\K_\fp^+$ if
$\K_\fp=\K_\fp^+$. This happens if and only if $\fp$ splits completely in
$\K$; that is, if and only if the complex conjugation automorphism $\tau$
does not lie in the decomposition group of a prime $\fP\subset\K$ lying above
$\fp$. Since the Galois group $\Gal(\K/\Q)$ is cyclic, $\tau$ is its unique
involution. Hence for $\K_\fp=\K_\fp^+$ to hold it is necessary and
sufficient that the decomposition group of $\fP$ has odd order; equivalently,
the inertia degree of $q$ in $\K/\Q$ is odd; that is, the order $\ord_p(q)$
is odd. Thus, if $\ord_p(q)$ is odd, then $n+\rho$ is a norm from $\K_\fp$ to
$\K_\fp^+$.

To complete the proof, we show that for $\ord_p(q)$ even, $n+\rho$ is a norm
from $\K_\fp$ to $\K_\fp^+$ if and only if $v_q(\frac{p-1}4)$ is also even.
So assume now that $\ord_p(q)$ is even. Since $\K_\fp/\K_\fp^+$ is an
unramified quadratic extension, by
 \cite[Chapter~V, \S2, Corollary to Proposition~3]{b:se}, the group of norms
from $\K_\fp$ to $\K_\fp^+$ inside $(\K_\fp^+)^\times$ is
 $\langle\pi_\fp^2\rangle\times U_{\K_\fp^+}$, where $\pi_\fp$ is a
uniformizer of $\K_\fp^+$ (i.e. $v_\fp(\pi_\fp)=1$) and $U_{\K_\fp^+}$ is the
unit group of $\K_\fp^+$. Thus, $n+\rho$ is a norm from $\K_\fp$ to
$\K_\fp^+$ if and only if $v_\fp(n+\rho)$ is even. Let
$\rho':=\frac{-\sqrt{p}-1}2$ be the conjugate of $\rho$ over $\Q$. Observe
that
  $$ 0=v_q(2n-1)=v_\fq(2n-1)=v_\fq(n+\rho+n+\rho')\ge
                                     \min\{v_\fq(n+\rho),v_\fq(n+\rho')\} $$
implies
\begin{equation}\label{e:Hasse1}
  \min\{v_\fq(n+\rho),v_\fq(n+\rho')\}=0,
\end{equation}
and also that
\begin{equation}\label{e:Hasse2}
  \textstyle
  v_q(\frac{p-1}{4})=v_{\fq}(\frac{p-1}{4})=v_{\fq}((n+\rho)(n+\rho'))
                                         =v_{\fq}(n+\rho)+v_{\fq}(n+\rho').
\end{equation}

If $v_q(\frac{p-1}4)$ is odd, then either $v_\fq(n+\rho)$ is odd, or
$v_\fq(n+\rho')=v_{\fq'}(n+\rho)$ is odd; hence, either $n+\rho$ is not a
norm from $\K_\fp$ to $\K_\fp^+$, or it is not a norm from $\K_{\fp'}$ to
$\K_{\fp'}^+$ for some prime $\fp'$ of $\K^+$ lying above $\fq'$. It follows
that if $v_q(\frac{p-1}4)$ is odd, then $n+\rho$ is not a norm from $\K$ to
$\K^+$. On the other hand, if $v_q(\frac{p-1}4)$ is even, then by
\refe{Hasse1} and \refe{Hasse2}, $v_\fq(n+\rho)$ is also even and, similarly,
$v_{\fq'}(n+\rho)=v_\fq(n+\rho')$ is even. Therefore if $v_q(\frac{p-1}4)$ is
even, then $n+\rho$ is a norm from $\K_\fp$ to $\K_\fp^+$ for all $\fp$ lying
above $q$.

This completes the proof.
\end{proof}

\bigskip

\end{document}